\documentclass[11pt]{article}
\usepackage{amscd}
\usepackage{amsmath}
\usepackage{amsfonts}

\newcommand{\de}{\partial}

\newcommand{\ve}{\varepsilon}

\newcommand{\Ric}{\mathrm{Ric}}
\newcommand{\ov}[1]{\overline{#1}}

\newcommand{\ti}[1]{\tilde{#1}}
\newcommand{\tr}[2]{\textrm{tr}_{#1} #2}

\begin{document}
\newcounter{remark}
\newcounter{theor}
\setcounter{remark}{0}
\setcounter{theor}{1}
\newtheorem{claim}{Claim}
\newtheorem{theorem}{Theorem}[section]
\newtheorem{proposition}{Proposition}[section]
\newtheorem{conjecture}{Conjecture}[section]
\newtheorem{question}{Question}[section]
\newtheorem{lemma}{Lemma}[section]
\newtheorem{defn}{Definition}[theor]
\newtheorem{corollary}{Corollary}[section]
\newenvironment{proof}[1][Proof]{\begin{trivlist}
\item[\hskip \labelsep {\itshape #1}]}{\hfill$\square$\medskip\end{trivlist}}
\newenvironment{remark}[1][Remark]{\addtocounter{remark}{1} \begin{trivlist}
\item[\hskip
\labelsep {\bfseries #1  \thesection.\theremark}]}{\end{trivlist}}
\setlength{\arraycolsep}{2pt}
\centerline{\bf THE CALABI-YAU EQUATION, SYMPLECTIC FORMS}
\centerline{\bf AND ALMOST COMPLEX STRUCTURES\footnote{The first author is supported in part by a Harvard Merit Fellowship. The second author is supported in part  by National Science Foundation grant  DMS-08-48193 and a Sloan Fellowship.}
}
%\addtocounter{section}{1}
\bigskip

$$\begin{array}{rlrl}
& \textbf{Valentino Tosatti} \qquad & & \textbf{Ben Weinkove} \\
 & \textrm{Mathematics Department} \qquad & & \textrm{Mathematics Department} \\
& \textrm{Harvard University} & & \textrm{University of California, San Diego} \\
& \textrm{Cambridge, MA 02138} & & \textrm{La Jolla, CA 92093}
\end{array}$$

\bigskip
\bigskip

\centerline{\em Dedicated to Professor S.-T. Yau on the occasion of his 60th birthday.}

\bigskip
\bigskip

\noindent
{\bf Abstract.}  We discuss a conjecture of Donaldson on a version of Yau's Theorem for symplectic forms with compatible almost complex structures and survey some recent progress on this problem.  We also speculate on some future possible  directions, and use a monotonicity formula for harmonic maps to obtain a new local estimate in the setting of Donaldson's conjecture.

\section{Background - Yau's Theorem}

In this section we give some background on Yau's Theorem \cite{Ya} in K\"ahler geometry, formerly known as the Calabi Conjecture.  It can be stated as follows.

\begin{theorem}\label{yau1}
Let $(M, \omega)$ be a compact K\"ahler manifold of complex dimension $n$.  If $\sigma$ is a volume form on $M$ satisfying $\int_M \sigma = \int_M \omega^n$ then there exists a unique K\"ahler form $\tilde{\omega}$  in $[\omega]$ satisfying
\begin{equation} \label{CalabiYau}
\tilde{\omega}^n = \sigma.
\end{equation}
\end{theorem}

The uniqueness part of the theorem was proved earlier by Calabi \cite{Ca1}.  We will call (\ref{CalabiYau}) the Calabi-Yau equation.  

Yau's Theorem shows that the space of K\"ahler forms in a fixed K\"ahler class $\beta$ can be identified with the space of volume forms on $M$ with integral $\beta^n$ via the map $\omega \mapsto \omega^n$.   Yau's Theorem can also be stated in terms of the first Chern class of the manifold.

\begin{theorem} \label{yau2}
Let $(M, \omega)$ be a compact K\"ahler manifold of complex dimension $n$.  If $\Psi$ is a closed real $(1,1)$-form representing the cohomology class $c_1(M)$ then there exists a unique K\"ahler metric $\tilde{\omega} \in [\omega]$ satisfying
\begin{equation} \label{CalabiYau2}
 \frac{1}{2\pi} \emph{Ric}(\tilde{\omega}) = \Psi.
 \end{equation}
\end{theorem}

It is not difficult to see that Theorems \ref{yau1} and \ref{yau2} are equivalent.  Indeed, assuming Theorem \ref{yau1} we proceed as follows.   The first Chern class $c_1(M)$ is represented by $\frac{1}{2\pi} \textrm{Ric}(\omega)$ and hence the $\partial \ov{\partial}$-Lemma produces a smooth function $F$ on $M$, which we may assume satisfies $\int_M e^F \omega^n = \int_M \omega^n$, 
with
\begin{equation} \label{Psi}
\Psi = \frac{1}{2\pi} \textrm{Ric}(\omega)  - \frac{\sqrt{-1}}{2\pi} \partial \ov{\partial} F.
\end{equation}
By the definition of the Ricci curvature, (\ref{CalabiYau2}) is then equivalent to
\begin{equation} \label{CalabiYau3}
\frac{\sqrt{-1}}{2\pi} \partial\ov{\partial} \log \frac{\tilde{\omega}^n}{e^F \omega^n} =0.
\end{equation}
Since every pluriharmonic function on $M$ is constant, one immediately sees that solving  (\ref{CalabiYau3}) is equivalent to finding a K\"ahler form $\tilde{\omega}$ in $[\omega]$ satisfying
$$\tilde{\omega}^n =  \sigma,$$
namely, equation (\ref{CalabiYau}), for $\sigma = e^F \omega^n$.  Conversely, given $\sigma =: e^F \omega^n$ as in Theorem \ref{yau1}, one can define $\Psi \in c_1(M)$ by (\ref{Psi}) and see in the same way that $\tilde{\omega}$ solving (\ref{CalabiYau2}) satisfies $\tilde{\omega}^n = \sigma$. 

An immediate corollary of Theorem \ref{yau2} is the following widely-used result, which is also sometimes referred to as Yau's Theorem.

\begin{corollary}
If a compact K\"ahler manifold $(M, \omega)$ satisfies $c_1(M)=0$ then there exists a unique K\"ahler form $\tilde{\omega} \in [\omega]$ with $\emph{Ric}(\tilde{\omega})=0$.
\end{corollary}

This result produces Ricci flat metrics on a large class of algebraic varieties, and this has had an enormous impact on algebraic geometry and string theory. For example they were used by Todorov \cite{Td} and Siu \cite{Si} to prove two long-standing conjectures about K3 surfaces.

We recall now the proof of Theorem \ref{yau1}.  Yau used a continuity method as follows.  Write $\sigma = e^F \omega^n$ and consider the 1-parameter family of equations
$$
(*)_t \qquad \qquad \qquad \tilde{\omega}_t^n = e^{tF+c_t} \omega^n, \quad t \in [0,1],$$
for constants $c_t$ defined by $e^{-c_t} = \int_M e^{tF} \omega^n/\int_M\omega^n$.  Clearly $\tilde{\omega}_0=\omega$ solves $(*)_t$ for $t=0$.  To solve $(*)_t$ for $t \in [0,1]$, Yau proved $C^{\infty}$ estimates on $\tilde{\omega}_t$ depending on the fixed data $M$, $\omega$, $F$.  Combining these estimates with an implicit function theorem argument shows that the set
$$\{ t \in [0,1] \ | \ (*)_t \textrm{ admits a smooth solution}\}$$
is open and closed in $[0,1]$ and hence equal to $[0,1]$.  The K\"ahler form $\tilde{\omega}=\tilde{\omega}_1$ then solves (\ref{CalabiYau}).

The $C^{\infty}$ estimates of Yau can be stated as: 

\pagebreak[3]
\begin{theorem}  \label{yauestimates}
Let $(M, \omega)$ be a compact K\"ahler manifold of complex dimension $n$.  If a K\"ahler form $\tilde{\omega} \in [\omega]$ solves the Calabi-Yau equation
$$\tilde{\omega}^n = \sigma,$$
for some volume form $\sigma$ on $M$ then there are $C^{\infty}$ a priori bounds on $\tilde{\omega}$ depending only on $\omega$, $M$ and $\sigma$.  

More precisely, we have the following.  For each $k=0,1,2, \ldots$, there exists a constant $A_k$ depending only on $M$, $\omega$, $\sigma$ (with smooth dependence on $\sigma$ and $\omega$) such that
$$\| \tilde{\omega} \|_{C^k(g)} \le A_k,$$
where $g$ is the K\"ahler metric associated to $\omega$.
\end{theorem}

Of course, we would take $\sigma = e^{tF+c_t}\omega^n$ in order to apply this theorem to the argument above.

Donaldson \cite{D} noted that the assertion of Theorem \ref{yau1} makes sense even if the complex structure is not integrable.  In that case, one can take $\omega$ to be a symplectic form compatible with an almost complex structure $J$ and seek a symplectic form $\tilde{\omega}$, cohomologous to $\omega$, satisfying the Calabi-Yau equation
$$\tilde{\omega}^n = \sigma,$$
for some given volume form.   It turns out that the equation $(\omega +da)^n = \sigma$ for a 1-form $a$ satisfying $d^*a=0$ is overdetermined for $n>2$ and so we restrict to the case $n=2$.  We also remark that we do not expect the analogue of Theorem \ref{yau2} to hold, due to the problem of finding a function $F$ solving (\ref{Psi}), see also Conjecture \ref{conjecture4} below.

Donaldson \cite{D} conjectured that in dimension 4, one could obtain $C^{\infty}$ bounds for solutions to the Calabi-Yau equation $\tilde{\omega}^2=\sigma$.  And, at least in the case when $b^+(M)=1$, he conjectured that the analogue of Theorem \ref{yau1} would hold. In \cite{We2} it was shown that the estimates all reduce to a $C^0$ bound on an `almost-K\"ahler potential' $\varphi$, and moreover, that $\| \varphi \|_{C^0}$ can be bounded (and hence the equation solved) in the case when the Nijenhuis tensor of the almost complex structure $J$ is suitably small.  

In fact, Donaldson described in \cite{D} a more general framework which includes a conjectural almost complex version of Yau's Theorem as a special case, with applications to symplectic forms and almost complex structures.    Further analytic results in the setting where the background symplectic form is only \emph{taming} the almost complex structure were given in \cite{TWY}, improving those of \cite{We2}.  We postpone the discussion of these estimates until Section \ref{estimates} below.

The outline of this survey paper is as follows.  In Section \ref{donconj}, we discuss Donaldson's conjecture and some applications to symplectic and almost complex geometry.  In Section \ref{estimates}, we discuss the estimates of \cite{TWY} and \cite{We2}.  In Section \ref{methods} we give a rough sketch of the main steps in the proof of Yau's estimates and explain how some are generalized in \cite{We2}, \cite{TWY}. In Section \ref{mono}, we describe how a monotonicity formula for harmonic maps can be applied to give a local estimate in the setting of Donaldson's conjecture.

\setcounter{equation}{0}
\section{Donaldson's conjecture and applications} \label{donconj}

In this section we discuss the conjecture of Donaldson on estimates for the Calabi-Yau equation and describe some consequences.  We begin by recalling some terminology.  A symplectic form $\omega$ on a manifold $M$ tames an almost complex structure $J$ if
$\omega(X,JX)>0$ for all nonzero tangent vectors $X$.  The symplectic form $\omega$ is compatible with $J$ if, in addition, 
$$\omega(JX, JY) = \omega(X,Y), \quad \textrm{for all } X, Y.$$
In either case, the data $(\omega, J)$ determines a Riemannian metric $g_{\omega}$ given by
$$g_{\omega} (X,Y) = \frac{1}{2} ( \omega(X,JY) + \omega(Y, JX)),$$
satisfying the almost-Hermitian condition $g_{\omega} (JX, JY) = g_{\omega} (X,Y)$.

In \cite{D}, Donaldson made the following conjecture on $C^{\infty}$ estimates of solutions of the Calabi-Yau equation in terms of a reference taming symplectic form.    His conjecture is restricted to the case of four real dimensions, for reasons that will be made clear later.

\begin{conjecture} \label{conjecture1}
Let $(M, \Omega)$ be a compact symplectic four-manifold equipped with an almost complex structure $J$ tamed by $\Omega$.  Let $\sigma$ be a smooth volume form on $M$.  If $\tilde{\omega} \in [\Omega]$ is a symplectic form on $M$ which is compatible with $J$ and solves the Calabi-Yau equation
\begin{equation} \label{CY1}
\tilde{\omega}^2 = \sigma,
\end{equation}
then there are $C^{\infty}$ a priori bounds on $\tilde{\omega}$ depending only on $\Omega$, $J$ and $\sigma$.  

More precisely, we have the following.  For each $k =0, 1, 2, \ldots$, there exists a constant $A_k$ depending smoothly on the data $\Omega$, $J$ and $\sigma$ such that 
\begin{equation} \label{Ck}
\| \tilde{\omega} \|_{C^k(g_{\Omega})} \le A_k.
\end{equation}
\end{conjecture}

Note that the estimate (\ref{Ck}) on $\tilde{\omega}$ for $k=0$ together with the equation (\ref{CY1}) immediately imply the additional estimate
$$\tilde{g} \ge c \, g_{\Omega},$$
for some uniform constant $c=c(\Omega, J, \sigma)>0$, where $\tilde{g}$ is the metric associated to $\tilde{\omega}$.

It is perhaps worth remarking that Conjecture \ref{conjecture1} would be false in general if the cohomological  condition $\tilde{\omega} \in [\Omega]$ were removed, even in the K\"ahler case (cf. \cite{D}, Section 3.3).  Indeed, suppose that $(M, \Omega, J)$ were a K\"ahler manifold admitting a sequence of K\"ahler classes $\beta_i$ satisfying $\beta_i^2 = \int_M \sigma$ with a non-K\"ahler limit $\beta_{\infty} = \lim_{i \rightarrow \infty} \beta_i$ in $H^{1,1}(M; \mathbb{R})$.  By Yau's theorem one could find a sequence of K\"ahler metrics $\tilde{\omega}_i \in \beta_i$ satisfying $\tilde{\omega}_i^2 = \sigma$.  If the estimates of Conjecture \ref{conjecture1} held in this case then one could take a subsequential limit of the $\tilde{\omega}_i$ to obtain a K\"ahler metric in $\beta_{\infty}$, a contradiction.  (For a discussion of a related problem of the behavior of Ricci-flat metrics as the K\"ahler class degenerates, see \cite{To2}).

We expect that one could replace the assumption $\tilde{\omega} \in [\Omega]$ with a weaker condition which would ensure that the cohomology class $[\tilde{\omega}]$ remains bounded and uniformly distant from the boundary of the K\"ahler cone in the K\"ahler case.  By the characterizations of the K\"ahler cone due to Buchdahl \cite{Bu}, Lamari \cite{La} and Demailly-Paun \cite{DP}, an element $\beta \in H^{1,1}(M; \mathbb{R})$ is K\"ahler if it is numerically positive on analytic cycles and if it is also a limit of K\"ahler classes.  In light of this it seems natural to ask:

\begin{question}
Can one replace the assumption $\tilde{\omega} \in [\Omega]$ in Conjecture \ref{conjecture1} with conditions on 
\begin{enumerate}
\item[(a)] the boundedness of $[\tilde{\omega}]$ in $H^{1,1}(M;\mathbb{R})$; and
\item[(b)] the data $[\tilde{\omega}] \cdot C$, for $J$-holomorphic curves $C$ in $M$? 
\end{enumerate}
\end{question}

Although we will see that \emph{applications} of Conjecture \ref{conjecture1} do require the restriction to dimension 4, we do not know any counterexample to the conjecture itself in higher dimensions.
 We pose as a question:

\begin{question}
Does Conjecture \ref{conjecture1} hold in any dimension?
\end{question}

We now describe an application of Conjecture \ref{conjecture1}.  First, recall the well-known fact that given a general almost complex four-manifold $(M, J)$ which admits symplectic forms there may not exist a symplectic form $\omega$ compatible with $J$.  
Donaldson \cite{D} conjectured that the (obviously necessary) condition of the existence of a taming symplectic form for $J$ is sufficient for the existence of a compatible $\omega$.  Combining Donaldson's conjecture with a characterization of the existence of taming symplectic forms due to Sullivan \cite{Su} we get:

\begin{conjecture} \label{conjecture2}
Let $(M, J)$ be a compact almost complex four-manifold with $b^+(M)=1$. Then the following are equivalent:
\begin{enumerate}
\item[(i)] There exists a symplectic form on $M$ compatible with $J$.
\item[(ii)] There exists a symplectic form on $M$ taming $J$.
\item[(iii)] There is no nonzero closed positive current on $M$ which is of type $(1,1)$ with respect to $J$ and is homologous to zero.
\end{enumerate}
\end{conjecture}

\begin{proof}[Proof that Conjecture \ref{conjecture1} implies Conjecture \ref{conjecture2}]

We clearly have that $(i)\Rightarrow (ii)$. The implication $(ii)\Rightarrow (iii)$ is also trivial: if $\Omega$ tames $J$ and $T$ is a nonzero null-homologous closed positive $(1,1)$ current, then
\begin{equation}\label{tamed}
0=\langle\Omega,T\rangle=\langle\Omega^{1,1},T\rangle>0,
\end{equation}
because $\Omega^{1,1}$ is positive definite.
The fact that $(iii)\Rightarrow(ii)$ is a theorem of Sullivan (Theorem III.2 in \cite{Su}). 

It remains to show that $(ii)\Rightarrow(i)$.   Following Donaldson's argument (see the description in \cite{We2}), 
 we fix $\Omega$ and a symplectic form taming $J$. We then choose $J_0$, an almost complex structure compatible with $\Omega$, and connect it to $J=J_1$ with a smooth path $J_t$, $0\leq t\leq 1$, of almost complex structures all tamed by $\Omega$. We then look for a symplectic form $\omega_t$ compatible with $J_t$ with
$
[\omega_t]\in[\Omega]$
and satisfying the Calabi-Yau equation
\begin{equation}\label{cyeqn}
\omega_t^2=\Omega^2.
\end{equation}
Setting $\omega_0=\Omega$ clearly solves this for $t=0$ and the set $\mathcal{T}$ of all $t\in[0,1]$ such that we have a solution $\omega_t$ is open by Proposition 1 of \cite{D}.  This openness argument crucially uses the assumption of four dimensions.  Note that since $b^+(M)=1$, the span of $[\Omega]$ in $H^2(M; \mathbb{R})$ is trivially a maximal positive subspace for the intersection form, and this ensures that we can solve (\ref{cyeqn}) for $\omega_t$ in the same cohomology class as $\Omega$.  This is the only part of the proof where we use the condition $b^+(M)=1$. 

Closedness of $\mathcal{T}$ follows from Conjecture \ref{conjecture1} together with the Ascoli-Arzel\`a theorem.  Thus we have a solution $\omega_1$ of (\ref{cyeqn}), a symplectic form compatible with $J_1=J$.
\end{proof}

\begin{remark} Gromov  had shown in \cite{Gr} that $(ii)$ implies $(i)$ holds in the special case of $M= \mathbb{P}^2$ when the symplectic form $\Omega$ is the standard one and $J$ is any almost complex structure tamed by $\Omega$.
\end{remark}

As pointed out in \cite{D}, Conjecture \ref{conjecture2} is interesting even in the case when $J$ is integrable.  Indeed, at least in the case $b^+(M)=1$,  one can use the result to give another proof of the following result of Miyaoka-Siu \cite{M}, \cite{Si} which does not use the classification of complex surfaces (there are already such proofs by Buchdahl \cite{Bu} and Lamari \cite{La}).

\begin{theorem} \label{theoremcomplexsurface}
If $M$ is a complex surface with $b^1(M)$ even then $M$ is K\"ahler.
\end{theorem}
\begin{proof}[Proof of Theorem \ref{theoremcomplexsurface}  assuming $b^+(M)=1$ and Conjecture \ref{conjecture2}]
%Standard results on complex surfaces imply that if $b^1(M)$ is even then $b^+(M)$ is odd (see Theorem IV.2.7 in \cite{BHPV})
A result of Harvey-Lawson (Theorem 26 and pag.185 in \cite{HL}) says that if $b^1(M)$ is even then we can find a real closed $2$-form $\Omega$ on $M$ such that $\Omega^{1,1}$ is positive definite. Then 
$$\Omega^2=(\Omega^{1,1})^2+2\Omega^{2,0}\wedge\overline{\Omega^{2,0}}$$
is a strictly positive $(2,2)$-form, hence $\Omega$ is a symplectic form taming $J$. Then Conjecture \ref{conjecture2} implies the existence of a symplectic form compatible with $J$, that is of a K\"ahler form.  Presumably, one ought to be able to remove the assumption $b^+(M)=1$ using appropriate generalizations of Conjecture \ref{conjecture1} and Conjecture \ref{conjecture2}.
\end{proof}

On the other hand, assuming the classification of surfaces (see \cite{BHPV}, for example) and Theorem 20 of \cite{HL}, it was shown by \cite{LZ} that:

\begin{theorem} \label{theoremLZ}
Conjecture \ref{conjecture2} holds in the case when $J$ is integrable, even when $b^+(M)>1$.
\end{theorem}

It should be noted that, despite this result, Conjecture \ref{conjecture1} is still open in the case when $J$ is integrable.  This may be somewhat surprising.  Indeed, if $\Omega$ tames an integrable $J$ as in the statement of Conjecture \ref{conjecture1} then by Theorem \ref{theoremLZ} one obtains a K\"ahler form $\omega$ and Yau's estimates show that $\tilde{\omega}$ can be bounded in terms of $\omega$ and $\sigma$.  
However, since  there are no estimates available on the K\"ahler form $\omega$ in terms of the data $(\Omega, J)$, this falls short of what is needed for Conjecture \ref{conjecture1}.

Also, as one can see from the proof that Conjecture \ref{conjecture1} implies Conjecture \ref{conjecture2}, if Conjecture \ref{conjecture1} were to hold for $J$ integrable it would not (at least by the same argument) give another proof of Theorem \ref{theoremLZ}. 

We now mention another consequence of Conjecture \ref{conjecture1}.

\begin{conjecture} \label{conjecture3}
Let $(M, \omega)$ be a compact symplectic  four-manifold with a compatible almost complex structure $J$.  Assume $b^+(M)=1$.  Then for any smooth volume form $\sigma$ on $M$ with $\int_M \sigma = \int_M \omega^2$ there exists a unique symplectic form $\tilde{\omega} \in [\omega]$ on $M$ compatible with $J$, solving the Calabi-Yau equation
\begin{equation} \label{CY2}
\tilde{\omega}^2 = \sigma.
\end{equation}
\end{conjecture}
 
We remark that the uniqueness part of Conjecture \ref{conjecture3} is already known to hold (cf. \cite{D} and also \cite{We2}).  Indeed, Donaldson proved the following stronger uniqueness result which does not require the assumption $b^+(M)=1$.  Fix a maximal positive subspace $H_2^+ \subset H^2(M; \mathbb{R})$ for the intersection form on $M$.  Then if $\tilde{\omega}_1, \tilde{\omega}_2 \in [\omega] + H_2^+$ satisfy 
\begin{equation} \label{CYuniqueness}
\tilde{\omega}_1^2 = \tilde{\omega}_2^2,
\end{equation}
it follows that $\tilde{\omega}_1 = \tilde{\omega}_2$.  Of course, the case $b^+(M)=1$ corresponds to taking $H_2^+$ to be the span of $[\omega]$.

 To prove this general uniqueness result we can argue as follows.  Since $[\tilde{\omega}_1] - [\tilde{\omega}_2] \in H_2^+$, we have
\begin{equation} \label{u1}
\int_M ( \tilde{\omega}_1 - \tilde{\omega}_2)^2 \ge 0.
\end{equation}
Using (\ref{CYuniqueness}), we can find a unitary frame $\theta_1, \theta_2$ with respect to $(\tilde{\omega}_1, J)$, at a fixed point $p$ in $M$, so that 
$$\tilde{\omega}_1 = \sqrt{-1}  \theta_1 \wedge \overline{\theta}_1 + \sqrt{-1} \theta_2 \wedge \overline{\theta}_2, \qquad \tilde{\omega}_2 = \sqrt{-1} \lambda \theta_1 \wedge \overline{\theta}_1 +  \frac{\sqrt{-1}}{\lambda}  \theta_2 \wedge \overline{\theta}_2,$$
for some positive constant $\lambda$.  Moreover, 
\begin{equation} \label{u2}
( \tilde{\omega}_1 - \tilde{\omega}_2)^2 = \tilde{\omega}_1^2 \left(2 - \left( \lambda + \frac{1}{\lambda} \right) \right) \le 0,
\end{equation}
with equality if and only if $\lambda =1$.  Then from (\ref{u1}) and (\ref{u2}) we obtain $\tilde{\omega}_1 = \tilde{\omega}_2$ as required.

We now explain how Conjecture \ref{conjecture3} follows from Conjecture \ref{conjecture1}.

\begin{proof}[Proof that Conjecture \ref{conjecture1} implies Conjecture \ref{conjecture3}]  This is contained in \cite{We2}, but we outline the proof here for the reader's convenience.
Write $\sigma=e^F\omega^2$ for some smooth function $F$. We then consider the Calabi-Yau equations
\begin{equation}\label{cyeqns}
\ti{\omega}_t^2=e^{tF+c_t}\omega^2,
\end{equation}
where $0\leq t\leq 1$, each $\ti{\omega}_t$ is a symplectic form compatible with $J$, with cohomology class
$[\ti{\omega}_t]=[\omega]$ and the constants $c_t$ are chosen so that the integrals of both sides of \eqref{cyeqns} match. 
Then we have the trivial solution $\ti{\omega}_0=\omega$ at $t=0$ and the set of all $t\in[0,1]$ such that we have a solution $\ti{\omega}_t \in [\omega]$ is open by Proposition 1 of \cite{D}. Then Conjecture \ref{conjecture1} together with the Ascoli-Arzel\`a theorem implies closedness, and so the existence of a solution for $t=1$. 
\end{proof}

\begin{remark}
Delan\"oe \cite{De} considered a related problem concerning the Calabi-Yau equation. He investigated
solutions of $\ti{\omega}^n=e^F\omega^n$, on an almost-K\"ahler manifold $(M, \omega, J)$ of
dimension $2n$, of the form $\ti{\omega} = \omega + d(Jd\varphi)$ for a smooth real function
$\varphi$ so that $\ti{\omega}$ tames $J$ (here $J$ acts on $1$-forms by duality). He showed that in real dimension 4, if there
exists such a solution for every smooth function $F$, then $J$ must be integrable. \end{remark}

Finally we consider the analogue of Theorem \ref{yau2}.  Suppose, as in Conjecture \ref{conjecture1}, that $M$ admits a symplectic form $\Omega$ taming an almost complex structure $J$.
Let $\nabla$ be an affine connection $M$.  We say that $\nabla$ is an \emph{almost-Hermitian connection} if
$$\nabla J = 0=\nabla g_\Omega.$$
It is well-known (see e.g. \cite{KN}) that almost-Hermitian connections always exist, and we will assume that $\nabla$ is one of them. 
Choose a local unitary frame $\{ e_1, \ldots, e_n \}$ for
$g_\Omega$, and let $\{ \theta^1, \ldots, \theta^n \}$ be a dual coframe. 
Then locally there exists a matrix of complex valued 1-forms $\{ \theta_i^j \}$, called the connection 1-forms, such that
$$\nabla e_i = \theta_i^j e_j.$$
Applying $\nabla$ to $g_\Omega(e_i, \ov{e_j})$ we see that $\{ \theta_i^j \}$ satisfies the skew-Hermitian property
$$\theta_i^j + \ov{\theta_j^i} = 0.$$
Now define the torsion $\Theta = (\Theta^1, \ldots, \Theta^n)$ of $\nabla$ by
\begin{equation} \label{eqnstructure1}
d \theta^i = - \theta_j^i \wedge \theta^j + \Theta^i, \qquad \textrm{for } i=1, \ldots, n.
\end{equation}
Notice that the $\Theta^i$ are 2-forms.  Equation (\ref{eqnstructure1}) is known as the first structure equation.
Define the curvature $\Psi = \{ \Psi_j^i \}$   of $\nabla$ by
\begin{equation} \label{eqnstructure2}
d \theta_j^i = - \theta_k^i \wedge \theta_j^k + \Psi_j^i.
\end{equation}
Note that $\{ \Psi_j^i \}$ is a  skew-Hermitian matrix of 2-forms.  Equation (\ref{eqnstructure2}) is known as the second structure equation.
Differentiating \eqref{eqnstructure2} we see that the real 2-form $\frac{\sqrt{-1}}{2\pi}\Psi_i^i$ is closed (here we are summing over $i$), and by Chern-Weil theory it represents the first Chern class $c_1(M,\Omega)$.  Associated to an almost-Hermitian manifold $(M, g_{\Omega}, J)$ is a unique \emph{canonical connection} $\nabla$ satisfying the conditions:
\begin{enumerate}
\item[(i)] $\nabla J = 0 = \nabla g_{\Omega}$.
\item[(ii)] The torsion $(\Theta^j)$, viewed as a $T'M$-valued 2-form, has vanishing $(1,1)$-part.
\end{enumerate}
We will denote by $\Ric(g_\Omega, J)$ the $2$-form $\sqrt{-1}\Psi_i^i$
computed with the canonical connection. In general it is not of type $(1,1)$, but if $(M, g_\Omega, J)$ is K\"ahler then $\Ric(g_\Omega,J)$ is just the standard Ricci form. We then have the following conjecture, which should be compared with Theorem \ref{yau2}.
\begin{conjecture}\label{conjecture4}
Let $(M,\Omega)$ be a compact symplectic four-manifold with $b^+(M)=1$ equipped with an almost complex structure $J$ tamed by $\Omega$. Then for every smooth function $F$ on $M$ there exists a unique symplectic form $\ti{\omega}\in[\Omega]$ compatible with $J$ satisfying
\begin{equation}\label{ricciforms}
\Ric(\ti{g},J)=\Ric(g_\Omega,J)+\frac{1}{2}d(JdF).
\end{equation}
\end{conjecture}
Notice that if $J$ is integrable then
$\frac{1}{2}d(JdF)=-\sqrt{-1}\partial\overline{\partial}f$, and if $(M,g_\Omega, J)$ is K\"ahler then the $\partial\overline{\partial}$-lemma implies that by varying $F$, the right hand side of \eqref{ricciforms} can be made equal to any representative of $2\pi c_1(M)$. We do not expect this to hold in general.
\begin{proof}[Proof that Conjecture \ref{conjecture1} implies Conjecture \ref{conjecture4}] We are free to add a constant to $F$ so that it satisfies
\begin{equation}\label{normaliz}
\int_M e^{\frac{F}{2}}dV_{g_\Omega}=\int_M\frac{\Omega^2}{2}.
\end{equation}
Since Conjecture \ref{conjecture1} implies Conjecture \ref{conjecture3}, we can find a unique $\ti{\omega}\in[\Omega]$ compatible with $J$ satisfying
\begin{equation}\label{monge}
\frac{\ti{\omega}^2}{2}=e^{\frac{F}{2}} dV_{g_\Omega}.
\end{equation}
This can be written locally in terms of the metrics $\tilde{g}$ and $g_{\Omega}$ as
\begin{equation}\label{mongeloc}
\det \ti{g}=e^{F}\det g_\Omega,
\end{equation}
and the computation to derive (3.16) in \cite{TWY} gives
\begin{equation}\label{dlog}
\frac{1}{2}d(JdF)=\frac{1}{2}d\left(Jd\log \frac{\det\ti{g}}{\det g_\Omega}\right)=\Ric(\ti{g},J)-\Ric(g_\Omega,J),
\end{equation}
as required. The uniqueness statement follows easily once one notices that conversely \eqref{dlog} and \eqref{normaliz} imply \eqref{mongeloc} and so also \eqref{monge}.
\end{proof}

\setcounter{equation}{0}
\section{Estimates for the Calabi-Yau equation} \label{estimates}

In this section we describe a number of estimates for the Calabi-Yau equation which make some progress towards Conjecture \ref{conjecture1}.

In \cite{TWY}, it was shown that Conjecture \ref{conjecture1} holds, in any dimension, assuming a positive curvature condition on the fixed metric $g_{\Omega}$. 
The key to this result is to work with a good choice of local frame, an important technique for these kinds of problems (cf. \cite{To1}).

As in the previous section we let $(M,\Omega)$ be a compact symplectic four-manifold, $J$ be an almost complex structure tamed by $\Omega$ and $g_\Omega$ be the associated almost-Hermitian metric. Let $\nabla$ be  the canonical connection of $(M, g_\Omega, J)$ and we define a modified curvature tensor $\mathcal{R}_{i \ov{j} k \ov{\ell}}$ as follows:
$$\mathcal{R}_{i \ov{j} k \ov{\ell}} = R^j_{i k \ov{\ell}} + 4 N^r_{\ov{\ell} \, \ov{j}} \ov{N^i_{\ov{r} \, \ov{k}}},$$
where $R^j_{i k \ov{\ell}}$ is the $(1,1)$-part of the curvature of $\nabla$ and $N^r_{\ov{\ell} \, \ov{j}}$ is the Nijenhuis tensor, which can also be viewed as the $(0,2)$-part of the torsion of $\nabla$.  In the case when the data $(g_{\Omega}, J)$ is K\"ahler, the tensor $\mathcal{R}_{i \ov{j} k \ov{\ell}}$ coincides with the usual curvature tensor.  We write $\mathcal{R} \ge 0$ if the modified curvature tensor is nonnegative in the Griffiths sense, that is, if 
$$\mathcal{R}_{i \ov{j} k \ov{\ell}} X^i \ov{X^j} Y^k \ov{Y^{\ell}} \ge 0, \qquad \textrm{for all (1,0) vectors } X, Y.$$

Then in \cite{TWY} it is shown that:

\begin{theorem} \label{TWY1}
If $\mathcal{R} (g_{\Omega}, J) \ge 0$ then Conjecture \ref{conjecture1} holds.  Moreover, the analogous conjecture holds for manifolds of any even dimension.
\end{theorem}

We note that this gives the first examples of non-K\"ahler manifolds for which Conjecture \ref{conjecture1} holds.  In the case when $M= \mathbb{P}^n$ and $(g_{\textrm{FS}}, J)$ is the Fubini-Study metric, we have
$$\mathcal{R}_{i \ov{j} k \ov{\ell}} ( g_{\textrm{FS}}, J) = (g_{\textrm{FS}})_{i \ov{j}}  (g_{\textrm{FS}})_{k \ov{\ell}} +  (g_{\textrm{FS}})_{i \ov{\ell}}  (g_{\textrm{FS}})_{k \ov{j}},$$
and hence the condition $\mathcal{R} \ge 0$ holds whenever the data $(g_{\Omega}, \Omega)$ is not too far from the Fubini-Study metric.  We note that such results cannot be obtained using Yau's theorem and an implicit function type argument, since we require the estimates to hold for \emph{all} volume forms $\sigma$.

We also remark that the proof of Theorem \ref{TWY1} does not make use of the condition $\tilde{\omega} \in [\Omega]$.  However, this does not contradict the discussion in Section 2 on the necessity of a cohomological assumption, since the nonnegativity of $\mathcal{R}$ must impose restraints on the topology of $M$.

\bigskip
We discuss now the general case in dimension $2n$, with no curvature assumptions.   Suppose we are in the setting of Conjecture \ref{conjecture1}, so that $\Omega$ is a symplectic form taming $J$ while $\ti{\omega} \in [\Omega]$ is a symplectic form compatible with $J$ and satisfying $$\ti{\omega}^n=\sigma.$$
Inspired by the K\"ahler case we define a function $\varphi$ by
\begin{equation}\label{phidef}
\ti{\Delta}\varphi= 2n-\tr{\ti{g}}{g_\Omega},
\end{equation}
together with the normalization $\sup_M\varphi=0$. This definition is well-posed since it easy to see (cf. (3.2) in \cite{TWY}) that
\begin{equation}\label{stokes}
\tr{\ti{g}}{g_\Omega}=2n\frac{\ti{\omega}^{n-1}\wedge\Omega}{\ti{\omega}^n},
\end{equation}
and thus $\tr{\ti{g}}{g_{\Omega}}$ has average $2n$ with respect to $\tilde{\omega}^n$.  Note that if $J$ were integrable, and $\Omega$, $\tilde{\omega}$ K\"ahler with respect to $J$ then $\varphi$ would correspond to the usual K\"ahler potential defined by
$$\tilde{\omega} = \Omega + \sqrt{-1} \partial \overline{\partial} \varphi, \quad \sup_M \varphi=0.$$

We have the following result \cite{TWY}:

\begin{theorem} \label{TWY2} Fix an arbitrary constant $\alpha>0$.  Then, with the notation given above, there are $C^{\infty}$ a priori bounds on $\ti{\omega}$  depending only on $\Omega$, $J$, $\sigma$, $\alpha$ and 
\begin{equation*}
I_{\alpha}(\varphi) :=\int_M e^{-\alpha \varphi} \Omega^n.
\end{equation*}
\end{theorem}

This reduces Conjecture \ref{conjecture1} to establishing a uniform bound on the quantity $I_{\alpha}(\varphi)$ for some sufficiently small $\alpha>0$.  In the setting where $J$ is integrable and $\Omega$, $\tilde{\omega}$ are K\"ahler forms, the quantity $I_{\alpha}(\varphi)$ is always uniformly bounded when $\alpha$ is small, by a very general result which is independent of the Calabi-Yau equation \cite{Ho}, \cite{Ti}.  This gives then in particular an alternative proof of Yau's theorem (also, cf. \cite{We1}).

Finally, we note that, as a consequence of the proof of Theorem \ref{TWY2}, we also have:

\begin{theorem}\label{TWY3} With the notation given above, there are $C^{\infty}$ a priori bounds on $\ti{\omega}$  depending only on $\Omega$, $J$, $\sigma$ and $
\| \tilde{\omega} \|_{C^0(g_{\Omega})}$.
\end{theorem}

That is, the $C^{\infty}$ bounds on $\tilde{\omega}$ for Conjecture \ref{conjecture1} follow from a $C^0$ bound on $\tilde{\omega}$.  In fact, the result of Theorem \ref{TWY3} is  already contained in \cite{We2} in the special case when $\Omega$ is compatible with $J$.  We also mention that Donaldson \cite{D} proved a related result that a $C^0$ bound on $\tilde{\omega}$ together with a BMO type estimate on $\tilde{\omega}$  is enough to give Conjecture \ref{conjecture1}.

\setcounter{equation}{0}
\section{Methods} \label{methods}

In this section we will briefly outline the key estimates of Yau (Theorem \ref{yauestimates}) and describe, informally, those arguments and estimates that still hold in the setting of Donaldson's conjecture.

Let $\tilde{\omega}$ solve the Calabi-Yau equation 
$$\tilde{\omega}^n = e^F \omega^n,$$
on a compact K\"ahler manifold $M$, for some smooth function $F$.
 Let $\varphi$ be the K\"ahler potential, defined by
$$\ti{\omega} = \omega + \sqrt{-1} \partial \ov{\partial} \varphi, \quad \int_M \varphi \, \omega^n=0,$$
The key steps in proving $C^{\infty}$ \emph{a priori} bounds on $\tilde{\omega}$ are as follows. 

\bigskip
\noindent\emph{Step 1.} \  The inequality
\begin{equation} \label{trggestimate}
\tr{g}{\ti{g}} \le C e^{A(\varphi - \inf_M \varphi)},
\end{equation}
holds for uniform constants $A, C$.

\bigskip
\noindent
\emph{Step 2.} \ The K\"ahler potential $\varphi$  satisfies $\| \varphi \|_{C^0} \le C$, for a uniform $C$.

\bigskip
\noindent
\emph{Step 3.} \ If $\| \tilde{\omega} \|_{C^0}$ is uniformly bounded, we have 
$\| \tilde{\omega} \|_{C^1(g)} \le C$, for a uniform $C$.

\bigskip
\noindent
\emph{Step 4.} Given a H\"older bound $\| \tilde{\omega} \|_{C^{\beta}(g)} \le C$ for some $\beta>0$, we have, for each $k=2, 3, \ldots,$, the estimates 
$\| \tilde{\omega} \|_{C^k(g)} \le A_k,$
for uniform $A_k$.

\bigskip

The proof of Step 1 uses the maximum principle and the key inequality:
\begin{equation} \label{key}
\tilde{\Delta} \log \tr{g}{\tilde{g}} \ge -C_1 \tr{\ti{g}}{g} -C_2,
\end{equation}
for uniform constants $C_1$ and $C_2$.  Observe that applying the Laplace operator of $\tilde{g}$ to the quantity $\tr{g}{\ti{g}}$ gives rise to three terms.  Ignoring first order derivatives for the moment, one sees that 
two derivatives landing on $\tilde{g}$ give a term involving the Ricci curvature of $\tilde{g}$, which can be controlled using the Calabi-Yau equation.  When two derivatives land on $g$ this gives the full curvature tensor of $g$, and the resulting term can be bounded by $(\tr{g}{\tilde{g}})(\tr{\tilde{g}}{g})$.  Finally, the first order derivatives give rise to a positive quantity
\begin{equation} \label{goodterm}
g^{i \ov{j}} \ti{g}^{p \ov{q}} \ti{g}^{k \ov{\ell}} \nabla_i \ti{g}_{k \ov{q}} \nabla_{\ov{j}} \ti{g}_{p \ov{\ell}},
\end{equation}
which can be used to control the negative term 
$-(| d \tr{g}{\tilde{g}} |^2_{\ti{g}})/ (\tr{g}{\ti{g}})^2$ produced from  differentiating the logarithm function.

Once (\ref{key}) is established, the estimate (\ref{trggestimate}) follows immediately from the maximum principle applied to the quantity $(\log \tr{g}{\tilde{g}}- A \varphi)$ for a constant $A$ chosen sufficiently large.  The point is that the K\"ahler potential $\varphi$ satisfies the equation
$$\tilde{\Delta} \varphi = 2n - \tr{\ti{g}}{g},$$
and so by choosing $A$ larger than $C_1$,  the bad term $-C_1 \tr{\ti{g}}{g}$ in (\ref{key}) can be replaced by a good positive term. Then using the Calabi-Yau equation one sees that the quantities $\tr{\ti{g}}{g}$ and $\tr{g}{\tilde{g}}$ are basically equivalent.

Note that the same inequality (\ref{trggestimate}) holds for other equations in K\"ahler geometry such as the equation for K\"ahler-Einstein metrics with negative Ricci curvature \cite{Ya}, \cite{Au}.

Step 2 was achieved using the celebrated Moser iteration method of Yau.  We illustrate the basic idea by describing how to obtain a uniform $L^2$ estimate of $\varphi$.  By the Calabi-Yau equation,  
\begin{equation} \label{e1}
\int_M \varphi (\omega^n - \ti{\omega}^n) \le C \int_M | \varphi | \omega^n. 
\end{equation}
On the other hand, since $\tilde{\omega} = \omega + \sqrt{-1} \partial \ov{\partial} \varphi$, 
\begin{eqnarray} \nonumber
  \int_M \varphi (\omega^n - \ti{\omega}^n) & = & - \int_M \varphi \sqrt{-1} \partial \ov{\partial} \varphi \wedge \sum_{i=0}^{n-1}(\omega^i\wedge\ti{\omega}^{n-1-i}) \\
& \ge &  \int_M \sqrt{-1} \partial \varphi \wedge \ov{\partial} \varphi \wedge \omega^{n-1}, \label{e2}
\end{eqnarray}
after integrating by parts. Combining (\ref{e1}) and (\ref{e2}) with the Poincar\'e inequality 
$$\int_M | \varphi|^2 \omega^n \le C \int_M | \nabla \varphi|^2 \omega^n,$$
gives $\| \varphi \|_{L^2(\omega)} \le C$.  This idea can then be extended by an iteration process to give $L^p$ estimates $\| \varphi \|_{L^p} \le C(p)$ by using the quantity $\varphi | \varphi |^a$ for $a > 0$ in the above calculation instead of $\varphi$ and applying the Sobolev inequality instead of the Poincar\'e inequality.  
 The $C^0$ bound of $\varphi$ then follows after checking that the constants $C(p)$ remain bounded as $p \rightarrow \infty$.  

Note that this method differs somewhat from Yau's original proof, which was rather more involved and made use of Step 1.  Alternative proofs of Step 2 have been given by Kolodziej \cite{Ko}, Blocki \cite{Bl} and 
also by the second author, where it was shown in \cite{We1} that (\ref{trggestimate}) implies a uniform estimate of the potential $\varphi$. 

Observe that once Steps 1 and 2 have been established, the Calabi-Yau equation implies immediately that the metrics $g$ and $\tilde{g}$ are uniformly equivalent.   
Step 3 then follows from a maximum principle argument applied to the third order derivatives of $\varphi$.  This computation was inspired by an estimate of Calabi \cite{Ca2}.  The idea is to compute the Laplace operator of $\ti{g}$ applied to a quantity $S$, the norm-squared of the tensor $\nabla_i \nabla_{\ov{j}} \nabla_k \varphi$, with the norm taken with respect to $\tilde{g}$.  A lengthy calculation gives:
$$\tilde{\Delta} S \ge - C_1 S - C_2.$$
One can then apply the maximum principle to $(S + A\, \tr{g}{\ti{g}})$ for a sufficiently large constant $A$, making  use of the fact that $\tilde{\Delta} \tr{g}{\ti{g}}$ contains the  positive term (\ref{goodterm}) which is equivalent to $S$.  This then gives the bound for $S$ as  required in Step 3.   Step 4  follows from standard elliptic estimates after differentiating the Calabi-Yau equation.  This completes the outline of Yau's estimates.

We now discuss how these estimates can be extended in the non-integrable case.  Assume that we are in the setting of Conjecture \ref{conjecture1}, so that $M$ is a compact 4-manifold equipped with an almost complex structure $J$ and $\Omega$ is a symplectic form taming $J$.  We have a symplectic form $\tilde{\omega} \in [\Omega]$ compatible with $J$ and satisfying the Calabi-Yau equation $\tilde{\omega}^2 = \sigma$ for some volume form $\sigma$ (in fact, much of what we say here carries over easily to any dimension).  

  It turns out that Step 1 holds:  we have the estimate
\begin{equation} \label{key2}
\tilde{\Delta} \log \tr{g_{\Omega}}{\tilde{g}} \ge -C_1 \tr{\ti{g}}{g_{\Omega}} -C_2.
\end{equation}
This was first proved in \cite{We2} in the case when $\Omega$ is compatible with $J$, using normal coordinates and careful estimates of the terms involving the Nijenhuis tensor.  In \cite{TWY} it was shown that (\ref{key2}) holds even if $\Omega$ only tames $J$.  The method of \cite{TWY}, simplifying the arguments in \cite{We2}, was to use the method of moving frames and the canonical connection, as described in Section \ref{donconj}.  From (\ref{key2}), the analogue of (\ref{trggestimate}) then follows immediately with the potential $\varphi$ defined by (\ref{phidef}).  

Moreover, it was shown in \cite{TWY} that under the assumption $\mathcal{R}(g_{\Omega}, J) \ge 0$ discussed in Section \ref{estimates}, we have the stronger inequality
\begin{equation} \label{key3}
\tilde{\Delta}  \tr{g_{\Omega}}{\tilde{g}} \ge  -C_2.
\end{equation}
Then (\ref{key3}) together with the Calabi-Yau equation and an iteration argument, starting with the  $L^1$ estimate on $\tr{g_{\Omega}}{\ti{g}}$, gives a uniform upper bound on the quantity $\tr{g_{\Omega}}{\ti{g}}$.

Step 2 cannot be carried out in the same way as in Yau's theorem due to the lack of a $\partial \ov{\partial}$-Lemma.  This seems to be the missing ingredient in a direct proof of Conjecture \ref{conjecture1} along these lines.  

For Step 3, it was shown in \cite{TWY} that the analogue of the third order estimate does indeed hold in this setting, although the computation is significantly more involved.  An alternative approach to Step 3 was carried out in \cite{We2}, in the case when $\Omega$ is compatible with $J$, using the method of Evans and Krylov \cite{Ev}, \cite{Kr} (see also \cite{Tr}).   This argument exploits the concavity of the $\log \det$ function and gives a H\"older bound on $\tilde{\omega}$.  While this is weaker than the estimate $\| \tilde{\omega} \|_{C^1(g_{\Omega})}$ obtained by the maximum principle, it is sufficient for the purpose of obtaining higher order estimates.

Step 4 follows from standard elliptic theory as in the K\"ahler case, and is discussed in \cite{D}, \cite{We2}, \cite{TWY}.

Returning to Step 2:  a Moser type iteration argument making use of Step 1 gives instead an estimate
$$- \inf_M \varphi \le C_{\alpha} + \log \left( \int_M e^{-\alpha \varphi} dV_{g_{\Omega}} \right)^{1/\alpha}$$
for any strictly positive $\alpha>0$.   Combining this result with Steps 1, 2 and 4 gives the proof of Theorem \ref{TWY2}.

\pagebreak[3]
\setcounter{equation}{0}
\section{A monotonicity formula} \label{mono}
In this section we will describe how a monotonicity formula for harmonic maps can be used to give a local estimate for solutions to the Calabi-Yau equation.

In general if $(M,J)$ is an 
almost complex manifold and $g$ is a Riemannian metric which satisfies
$g(X,Y)=g(JX,JY)$ for all $X,Y$, then we can define a real $2$-form $\omega$, not necessarily closed, 
by setting 
\begin{equation}\label{11}
\omega(X,Y)=g(JX,Y). 
\end{equation}
In this case we call the data $(M,g,\omega,J)$ an almost-Hermitian manifold. 
Let us recall briefly the notion of a harmonic map. If $f:(M,g)\to (M',g')$
is a mapping between Riemannian manifolds, its differential $df$ can be viewed as a section of $T^*M\otimes f^*TM'$. This bundle has a natural connection induced from the Levi-Civita connections of $g$ and $g'$, and we define the Laplacian of the map $f$ to be
$\Delta f=\tr{g}{(\nabla df)},$ which is a section of $f^*TM'$. If we pick local coordinates $\{x^\alpha\}$ on $M$ 
and $\{y^i\}$ on $M'$ then, writing $f$ in components $\{f^i\}$, we have
\begin{equation}\label{maplap}
(\Delta f)^i=g^{\alpha\beta}\frac{\de^2 f^i}{\de x^\alpha\de x^\beta}-
g^{\alpha\beta}\Gamma^\gamma_{\alpha\beta}\frac{\de f^i}{\de x^\gamma}+
g^{\alpha\beta}\Gamma'^i_{jk}\frac{\de f^j}{\de x^\alpha}\frac{\de f^k}{\de x^\beta},
\end{equation}
where $\Gamma^\gamma_{\alpha\beta}$ and $\Gamma'^i_{jk}$ are the Christoffel symbols of $g$ and $g'$ respectively. A map $f$ is called \emph{harmonic} if $\Delta f=0$.
We have the following result of Lichnerowicz \cite{Li}.

\pagebreak[3]
\begin{theorem}\label{harmo}
Let $(M,g,\omega,J)$ and $(M',g',\omega',J')$ be two almost-Hermitian manifolds of real dimension $2n$ and $2n'$ respectively, 
and $f:M\to M'$ be a $(J,J')$-holomorphic map, that is a map that satisfies
$$df\circ J=J'\circ df.$$
If 
\begin{equation}\label{conditions}
d(\omega^{n-1})=(d\omega')^{2,1}=0,
\end{equation}
then $f$ is harmonic.
\end{theorem}
Notice that \eqref{conditions} is satisfied if $\omega$ and $\omega'$ are closed (the converse is also true if $n=2$), but the theorem fails for general almost-Hermitian manifolds that do not satisfy the assumption \eqref{conditions} (see (9.11) in \cite{EL}). As an aside, we note here that harmonic maps between Riemannian manifolds satisfy a Schwarz lemma \cite{GH} and so do holomorphic maps between K\"ahler manifolds \cite{Ya2}. For a general Schwarz lemma on holomorphic maps between almost-Hermitian manifolds, see \cite{To1}.

We now consider 
 the setting of Conjecture \ref{conjecture1} (in any dimension $2n$) and derive an equation for the Laplacian of the identity map from $M$ to itself with respect to two different metrics on $M$.
Note that the symplectic form $\Omega$ is only taming $J$ and so the $2$-form associated to $g_\Omega$ and $J$ by \eqref{11} is not $\Omega$ but rather its $(1,1)$-part $\widehat{\Omega}$.  For convenience, we will from now on denote $g_{\Omega}$ by $g$.

We would like to apply Theorem \ref{harmo} to the identity map $I:(M,g,\widehat{\Omega},J)\to(M,\ti{g},\ti{\omega},J)$, but because $\widehat{\Omega}^{n-1}$ is not closed in general we cannot do this directly.
However, $\ti{\omega}$ is compatible with $J$ and Lichnerowicz's proof of Theorem \ref{harmo} \cite{Li} shows that the last term on the right hand side of \eqref{maplap} is independent of $\tilde{g}$. Taking $f=I$, we see that $\Delta I$ can be uniformly bounded by quantities depending only on the metric $g$.

We will use this to derive a monotonicity formula, analogous to that of Price \cite{P}.
Let $\xi$ be any smooth vector field on $M$ and $u$ be any smooth map from $M$ to itself. We recall that the energy density of $u$ is the quantity
$$e(u)=\tr{g}{(u^*\ti{g})}=g^{ij}\frac{\de u^k}{\de x^i}\frac{\de u^\ell}{\de x^j}\ti{g}_{kl}.$$
Then integration by parts gives
\begin{equation}\label{stat}
\frac{1}{2}\int_M \mathrm{div}\xi\ e(u)dV_g-\int_M g^{ij}\xi^k_{,i}
\frac{\de u^\ell}{\de x^k}\frac{\de u^p}{\de x^j}\ti{g}_{\ell p}dV_g=
\int_M (\Delta u)^i \xi^j \frac{\de u^k}{\de x^j}\ti{g}_{ik} dV_g.
\end{equation}
Whenever $u$ is harmonic, the right hand side of \eqref{stat} vanishes and the equation says that $u$ is a critical point of the Dirichlet integral when we reparametrize the domain $M$ by diffeomorphisms. Such maps satisfy a monotonicity formula \cite{P}. 
If we now consider the case when $u$ is the identity map, we see that it is not necessarily harmonic, but the right hand side of \eqref{stat} is given by
\begin{equation}\label{errorterm}
\int_M (\Delta I)^i \xi^k \ti{g}_{ik} dV_g,
\end{equation}
with $(\Delta I)^i$ uniformly bounded in terms of $g$.
The energy density of $I$ is  $\tr{g}{\ti{g}}$. The monotonicity formula of Price then does not apply directly, but we can trace through its proof (we will follow the proof of Theorem 1 in \cite{GB}) to obtain the following result. 

\begin{theorem} Let $(M,\Omega)$ be a compact $2n$-dimensional symplectic manifold, $J$ an almost complex structure tamed by $\Omega$ and $\ti{\omega}$ another symplectic form compatible with $J$.
We also let $g$ and $\ti{g}$ be the associated Riemannian metrics. Then there exist constants $r_0,A>0$ that depend only on $(M,g)$ such that given any $p\in M$ and any $0<r<\rho<r_0$ we have
\begin{equation}\label{monot}
\frac{e^{Ar}}{r^{2n-2}}\int_{B_g(p,r)}  \emph{tr}_g{\ti{g}} \, dV_g\leq
\frac{e^{A\rho}}{\rho^{2n-2}}\int_{B_g(p,\rho)} \emph{tr}_{g}{\ti{g}} \,dV_g,
\end{equation}
where $B_g(p,r)$ denotes the geodesic ball in the metric $g$ centered at $p$ of radius $r$.
\end{theorem}
The reason why this holds is the following. Using partitions of unity we can assume that the domain is a ball in $\mathbb{R}^{2n}$ and the metric $g$ is close to being Euclidean. With the notation of \cite{GB},  we take $\xi$ to be the radial vector field multiplied by a cutoff function $\eta$. We substitute this into \eqref{stat} and make use of \eqref{errorterm}. Comparing with the proof of Theorem 1 in \cite{GB} the only new term that appears is the quantity \eqref{errorterm} which can be bounded by
$$C\int_{B_{(1+s)r}}\eta r \tr{g}{\ti{g}} \, dV_g.$$
and this can be absorbed into another term of the same kind in \cite{GB}.  This proves \eqref{monot}.

We now restrict to the $4$-dimensional case $n=2$ and we assume that $\ti{\omega}$ is cohomologous to $\Omega$ and satisfies the Calabi-Yau equation \eqref{CY1}. Then using  \eqref{CY1} we see that 
\begin{equation}\label{equiv}
C^{-1}\tr{g}{\ti{g}}\leq \tr{\ti{g}}{g}\leq C\tr{g}{\ti{g}},
\end{equation}
for a uniform constant $C$. Moreover \eqref{stokes} and Stokes' Theorem imply that the $L^1$ norm of $\tr{\ti{g}}{g}$ is uniformly bounded
\begin{equation}\label{l1bound}
\int_M \tr{\ti{g}}{g} \, dV_g\leq C\int_M \tr{\ti{g}}{g}\ \Omega^2
\leq C\int_M \ti{\omega}\wedge\Omega=C\int_M\Omega^2\leq C,
\end{equation}
and so \eqref{equiv}, \eqref{l1bound} together with the monotonicity formula \eqref{monot} give the following corollary.

\begin{corollary}
Using the notation as above, with $\tilde{g}$ solving the Calabi-Yau equation, there exists a uniform constant $C$ such that
\begin{equation}\label{decay}
\int_{B_g(p,r)}\emph{tr}_{g}{\ti{g}} \, dV_g\leq C r^2,
\end{equation}
for all $p\in M$ and $r>0$ small.
\end{corollary}

By analogy with the theory of harmonic maps we expect the following $\ve$-regularity result:
\begin{conjecture}\label{epsilo}
Let $(M,\Omega)$ be a compact symplectic four-manifold equipped with an almost complex structure $J$ tamed by $\Omega$. Let $\ti{\omega}$ be another symplectic form cohomologous to $\Omega$ and compatible with $J$. Given a smooth volume form $\sigma$, we assume that $\ti{\omega}$ satisfies the Calabi-Yau equation
$$\ti{\omega}^2=\sigma.$$
Then there exist constants $\ve, C, r_0>0$ that depend only on $\Omega$, $J$ and $\sigma$ such that if $$\frac{1}{r^2}\int_{B_g(p,r)}\emph{tr}_{g}{\ti{g}} \,dV_g\leq\ve,$$
for some $p\in M$ and some $0<r<r_0$, then 
$$\sup_{B_g(p,r/2)}\emph{tr}_{g}{\ti{g}}\leq \frac{C}{r^4}\int_{B_g(p,r)}\emph{tr}_{g}{\ti{g}}\, dV_g.$$
\end{conjecture}
Such a result holds for harmonic maps \cite{Sc} so one may wonder why it cannot just be applied directly in this case. The point is that a crucial step in the proof of the $\ve$-regularity in \cite{Sc} is the differential inequality
$$\Delta \tr{g}{\ti{g}}\geq -C_0 \tr{g}{\ti{g}} -C_1 (\tr{g}{\ti{g}})^2,$$
where $\Delta$ is the Laplacian of $g$, the constant $C_0$ depends on the Ricci curvature of $g$ while $C_1$ depends on the whole Riemann curvature tensor of $\ti{g}$. In the setting of the Calabi-Yau equation this is not controlled, and we are forced to use the Laplacian of $\ti{g}$ instead. The computation
$$\ti{\Delta} \tr{g}{\ti{g}}\geq -C_2 -C_3 (\tr{g}{\ti{g}})^2,$$
appears in \cite{We2}, or (3.19) of \cite{TWY}, where now $C_2$ and $C_3$ only depend on the fixed data. But since the Sobolev constant of $\ti{g}$ is not bounded a priori, the strategy of proof in \cite{Sc} breaks down.

If Conjecture \ref{epsilo} were proved, then together with \eqref{decay} it would strongly suggest that the blow-up set of a family of Calabi-Yau equations 
has real codimension at least $2$. It is tempting to speculate that this set should actually be represented by a $J$-holomorphic curve, see \cite{D}, and that this might ultimately lead to a proof of Conjecture \ref{conjecture1}. Results roughly along these lines have been proved by Taubes \cite{Ta} for solutions of the Seiberg-Witten equations, which exhibit less nonlinearity than the Calabi-Yau equation.

\bigskip
\bigskip
\noindent
{\bf Acknowledgements}
The authors would like to thank Professor Yau for many useful conversations, and for his advice and support.

\end{document}